\documentclass{amsart}
\usepackage{mathtools, microtype, mathrsfs, xfrac, hyperref, amssymb, enumerate, xspace, stmaryrd}
\usepackage[utf8]{inputenc}
\usepackage{tikz-cd}
\usepackage{pdflscape}
\usepackage{bm}

\usepackage{todonotes}

\numberwithin{equation}{section}

\numberwithin{subsection}{section}

\allowdisplaybreaks[1]

\newtheorem*{namedtheorem}{\theoremname}
\newcommand{\theoremname}{testing}

\newtheorem{theorem}{Theorem}
\newtheorem{proposition}[theorem]{Proposition}
\newtheorem{proposition-definition}[theorem]
{Proposition-Definition}
\newtheorem{corollary}[theorem]{Corollary}
\newtheorem{lemma}[theorem]{Lemma}
\newtheorem*{theorem*}{Theorem}

\theoremstyle{definition}
\newtheorem{definition}[theorem]{Definition}

\newtheorem{remark}[theorem]{Remark}

\newtheorem{question}{Question}
\newtheorem*{question*}{Question}

\theoremstyle{remark}

\renewcommand{\mathcal}{\mathscr}

 \newcommand\cB{\mathcal{B}}

\newcommand\cG{\mathcal{G}}

 \newcommand\cX{\mathcal{X}}
 \newcommand\cZ{\mathcal{Z}}

\renewcommand\AA{\mathbb{A}} 
\newcommand\CC{\mathbb{C}}

 \newcommand\PP{\mathbb{P}}
\newcommand\QQ{\mathbb{Q}} \newcommand\RR{\mathbb{R}}

 \newcommand\ZZ{\mathbb{Z}}

\newcommand\bC{\mathbf{C}} \newcommand\bD{\mathbf{D}}

\newcommand\bI{\mathbf{I}}

\newcommand\bO{\mathbf{O}} 
 
 \newcommand\bT{\mathbf{T}}

\newcommand\rI{\mathrm{I}}

 \newcommand\rR{\mathrm{R}}

 \newcommand\bfx{\mathbf{x}}

\newcommand\arr{\ifinner\to\else\longrightarrow\fi}

\newcommand\arrto{\ifinner\mapsto\else\longmapsto\fi}

\def\displaytimes_#1{\mathrel{\mathop{\times}\limits_{#1}}}

\def\displayotimes_#1{\mathrel{\mathop{\bigotimes}\limits_{#1}}}

\newcommand\aut{\operatorname{Aut}}

\newcommand\spec{\operatorname{Spec}}

\newcommand{\proj}{\operatorname{Proj}}

\newlength{\ignora}

\renewcommand{\setminus}{\smallsetminus}

\newcommand{\GL}{\mathrm{GL}}
\newcommand{\SL}{\mathrm{SL}}
\newcommand{\PGL}{\mathrm{PGL}}

\newcommand{\br}{\operatorname{Br}}

\DeclareFontFamily{U}{mathx}{\hyphenchar\font45}
\DeclareFontShape{U}{mathx}{m}{n}{
      <5> <6> <7> <8> <9> <10>
      <10.95> <12> <14.4> <17.28> <20.74> <24.88>
      mathx10
      }{}
\DeclareSymbolFont{mathx}{U}{mathx}{m}{n}
\DeclareFontSubstitution{U}{mathx}{m}{n}
\DeclareMathAccent{\widecheck}{0}{mathx}{"71}
\DeclareMathAccent{\wideparen}{0}{mathx}{"75}

\renewcommand{\epsilon}{\varepsilon}

\newcommand{\cha}{\operatorname{char}}

\newcommand{\ST}{\operatorname{ST}}
\newcommand{\thickslash}{\mathbin{\!\!\pmb{\fatslash}}}

\begin{document}

\title{The arithmetic of tame quotient singularities in dimension $2$}

\author{Giulio Bresciani}

\begin{abstract}
    Let $k$ be a field, $X$ a variety with tame quotient singularities and $\tilde{X}\to X$ a resolution of singularities. Any smooth rational point $x\in X(k)$ lifts to $\tilde{X}$ by the Lang-Nishimura theorem, but if $x$ is singular this might be false. 
    
    For certain types of singularities the rational point is guaranteed to lift, though; these are called singularities of type $\mathrm{R}$. This concept has applications in the study of the fields of moduli of varieties and yields an enhanced version of the Lang-Nishimura theorem where the smoothness assumption is relaxed.

    We classify completely the tame quotient singularities of type $\mathrm{R}$ in dimension $2$; in particular, we show that every non-cyclic tame quotient singularity in dimension $2$ is of type $\mathrm{R}$, and most cyclic singularities are of type $\mathrm{R}$ too.
\end{abstract}

\address{Scuola Normale Superiore\\Piazza dei Cavalieri 7\\
56126 Pisa\\ Italy}
\email{giulio.bresciani@gmail.com}

%\date{\arrday}

\maketitle
%~ \section{Introduction}

We work over a field $k$ with algebraic closure $\bar{k}$. A variety over $k$ is a geometrically integral, separated scheme of finite type over $k$. A variety over $k$ has tame quotient singularities if, étale locally, it is a quotient of a smooth variety by the action of a finite group of order prime to $\cha k$. Varieties with tame quotient singularities always have a resolution \cite[\S 6.1]{giulio-angelo-moduli}. A tame quotient singularity over $k$ is a pair $(S,s)$ where $S$ is a variety with tame quotient singularities and $s\in S(k)$ is a rational point. In this paper, we only deal with singularities which are quotient and tame, hence we often simply call $(S,s)$ a singularity. 

Let $\tilde{S}\to S$ be a resolution of singularities. If $s$ lifts to a rational point of $\tilde{S}$, we say that the singularity $(S,s)$ is \emph{liftable}. By the Lang-Nishimura theorem, this definition does not depend on the chosen resolution. 

\begin{question}\label{question}
	Can we find \emph{geometric} conditions on $(S,s)$, i.e. only depending on $(S_{\bar{k}},s_{\bar{k}})$, which force the singularity to be liftable?
\end{question}

We started investigating Question~\ref{question} in our recent joint paper with A. Vistoli \cite{giulio-angelo-moduli}, where we have shown that it has strong connections with the classical problem of whether a variety (possibly with additional structure, such as a marked point or a polarization) is defined over its field of moduli. 

For instance, we have re-proved Shimura's classical result \cite{shimura} that a generic principally polarized abelian variety of odd dimension is defined over its field of moduli using the fact that twisted forms of the singularity $(\AA^{d}/\pm1,[0])$ are always liftable for $d$ odd \cite[Corollary 6.25]{giulio-angelo-moduli}.

Let us recall a definition from \cite{giulio-angelo-moduli}. A singularity $(S,s)$ over $k$ is of type $\rR$ if, for every extension $K/k$, every subfield $k'\subset K$ and every singularity $(U,u)$ over $k'$ such that $(S_{K},s_{K})$, $(U_{K},u_{K})$ have a common étale neighbourhood, then $(U,u)$ is liftable over $k'$. In short, a singularity is of type $\rR$ if any twisted form over any field is liftable. Question~\ref{question} asks to characterize geometrically the singularities of type $\rR$.

\subsection*{The case for singularities of type $\rR$}

As we mentioned above, singularities of type $\rR$ are tightly connected with the study of fields of moduli. The following is a direct consequence of our joint work with Vistoli \cite[Theorem 5.4, Lemma 5.9]{giulio-angelo-moduli}.

\begin{theorem}[{Bresciani--Vistoli \cite{giulio-angelo-moduli}}]\label{thm:point}
	Let $X$ be a proper variety over $\bar{k}$ and $p\in X$ a smooth point. Assume that $\aut(X,p)$ is finite étale and tame. If the singularity $(X/\aut(X,p),[p])$ is of type $\rR$, then $(X,p)$ is defined over its field of moduli.
\end{theorem}

Applying Theorem~\ref{thm:point} to a generic, odd-dimensional principally polarized abelian variety yields a more conceptual proof of Shimura's result \cite{shimura}. Varieties with marked points are just an example of the applications of $\rR$-singularities to the study of fields of moduli, see \cite{giulio-fmod} for others.

Recall that the classical Lang--Nishimura theorem states that, if $X\dashrightarrow Y$ is a rational map of varieties over a field $k$, $p\in X(k)$ is a smooth rational point and $Y$ is proper, then $Y(k)$ is non-empty. If $p$ is not smooth, then this is false: for instance, the hypersurface $X\subset \PP^{3}_{\QQ}=\proj \QQ[x,y,z,w]$ defined by $x^{2}+y^{2}+z^{2}=0$ has a unique rational point $p=(0:0:0:1)$, but the blowup of $X$ in $p$ has no rational points.

The following enhanced version of Lang--Nishimura is a direct consequence of the definition of $\rR$-singularities.

\begin{theorem}[{Lang--Nishimura, enhanced}]\label{thm:LN+}
	Let $X\dashrightarrow Y$ be a rational map of varieties over a field $k$ with $Y$ proper. Suppose that $p\in X(k)$ is a rational point and that $(X,p)$ is of type $\rR$. Then $Y(k)$ is non-empty.
\end{theorem}

Furthermore, the Lang-Nishimura theorem holds for tame stacks (in characteristic $0$, these are just Deligne--Mumford stacks) \cite[Theorem 4.1]{giulio-angelo-valuative}. Using it, it is immediate to prove the following.

\begin{theorem}[{Bresciani--Vistoli \cite{giulio-angelo-valuative}}]\label{thm:moduli}
	Let $\cX$ be an integral, regular tame stack of finite type over $k$ with coarse moduli space $X$, suppose that $\cX$ is generically a scheme, e.g. the moduli stack of curves of genus $g\ge 3$ in characteristic $0$. Let $p\in X(k)$ be a rational point. If the singularity $(X,p)$ is of type $\rR$, then $p$ lifts to a rational point of $\cX$.
\end{theorem}

The more singularities of type $\rR$ we find, the stronger Theorems~\ref{thm:point}, \ref{thm:LN+} and \ref{thm:moduli} become. Singularities of type $\rR$ are more common than one might expect, and we will show that most tame quotient singularities in dimension $2$ are of type $\rR$.

J.-L. Colliot-Thélène pointed to us that, when the base field is $\RR$, liftable points are called \emph{central points} and play a role in real algebraic geometry, e.g. a version of the Nullstellensatz over $\RR$ naturally involves them \cite[Corollary 7.6.6]{bochnak-coste-roy}. If $X$ is a variety of dimension $d$ over $\RR$ with smooth locus $X^{\rm sm}$ and $x\in X(\RR)$ is a real point, then $x$ is liftable, or central, if and only if it is in the closure of $X^{\rm sm}(\RR)$, if and only if the local dimension of $X(\RR)$ in $x$ is $d$ \cite[Proposition 7.6.2]{bochnak-coste-roy}. Studying singularities of type $\rR$ thus gives a way of finding central points on a real variety.

\subsection*{Known and new results}

A variety with quotient singularities is normal, hence in dimension $1$ every tame quotient singularity is regular and of type $\rR$. As we have mentioned above, in \cite{giulio-angelo-moduli} we have proved that $(\AA^{d}/\pm1,[0])$ is of type $\rR$ for $d$ odd. Furthermore, we have proved that if $G\subset\GL_{d}$ is a finite subgroup of order prime with $d!\cdot\cha k$, then $(\AA^{d}/G,[0])$ is of type $\rR$. Here, we classify completely the tame quotient singularities of type $\rR$ in dimension $2$. 

Let $K$ be a separable closure of $k$ and $1\le d\le n$ integers with $\gcd(n,d)=\gcd(n,\cha K)=1$, write $\zeta_{n}\in K$ for a primitive root of the unity and let $M$ be a diagonal $2\times 2$ matrix with eigenvalues $\zeta_{n},\zeta_{n}^{d}$. A $2$-dimensional singularity $(S,s)$ over $k$ is of type $\frac1n(1,d)$ if $(S_{K},s_{K})$ and $(\AA^{2}/<M>,[0])$ have a common étale neighbourhood. A tame quotient singularity is cyclic if, étale locally, it is the quotient of a smooth variety by a cyclic group of order prime with $\cha k$; every $2$-dimensional cyclic singularity is of type $\frac 1n(1,d)$ for some $n,d$ \cite[Corollary 6.4]{giulio-angelo-moduli}.

\begin{theorem}\label{thm:R2sing}
	Let $(S,s)$ be a tame quotient singularity over a separably closed field $K$. The following are equivalent.
	\begin{itemize}
			\item $(S,s)$ is not of type $\rR$.
			\item $(S,s)$ is of type $\frac1n(1,d)$, where $1\le d\le n$ are positive integers such that $n$ is even, $\cha k$ does not divide $n$ and, for every prime power $\ell^{a}$ dividing $n$, $d$ is congruent either to $1$ or $-1$ modulo $\ell^{a}$.
	\end{itemize}
\end{theorem}

In short, every non-cyclic singularity and most cyclic singularities are of type $\rR$ in dimension $2$.

In \cite{giulio-p2} \cite{giulio-points}, we apply Theorem~\ref{thm:R2sing} to study fields of moduli of curves and finite subsets in $\PP^{2}$.

\begin{remark}
	Quotient singularities are Kawamata log-terminal, and the converse holds in dimension $2$ and characteristic $0$ \cite[Remark 0-2-17]{kmm}. Hence, in characteristic $0$ we are working with Kawamata log-terminal singularities.
	
	As a consequence of Theorem~\ref{thm:R2sing} it is possible to find canonical singularities which are not of type $\rR$, namely $\frac1n(1,n-1)$ for $n$ even and prime with $\cha k$. More precisely, if $n$ is even and prime with $\cha k$, and there exists a Brauer class of index $2$ over $k$, then there exists a non-liftable twisted form of $\frac 1n(1,n-1)$, see Lemma~\ref{lem:non-lift}.
\end{remark}

\subsection*{Groups of type $\rR_{2}$}

Let $p\ge 0$ be either a prime or $0$, and $d\ge 1$ a positive integer. Let $G$ be a finite group, assume that $p$ does not divide the order of $G$. The group $G$ is of type $\rR_{d}$ in characteristic $p$ \cite[Definition 6.12]{giulio-angelo-moduli} if, for every faithful $d$-dimensional representation $G\subset\GL_{d}$ in characteristic $p$, the singularity $(\AA^{d}/G,[0])$ is of type $\rR$.

Clearly, if there are no faithful $d$-dimensional representations of $G$, then $G$ is of type $\rR_{d}$, the interesting cases are those in which a faithful representation exists. Using this definition, in order to check the hypotheses of Theorems~\ref{thm:point}, \ref{thm:LN+} and \ref{thm:moduli} it is enough to check that some group is of type $\rR_{d}$ where $d$ is the dimension of $X$ in $p$, namely $\aut(X,p)$ in Theorem~\ref{thm:point}, the local fundamental group of $X$ in $p$ in Theorem~\ref{thm:LN+} and the automorphism group of a geometric point over $p$ in Theorem~\ref{thm:moduli}. 

In \cite[Theorems 6.18, 6.19]{giulio-angelo-moduli} we have shown that there are infinitely many finite groups which are $\rR_{d}$ for every $d$, and that if $G$ is a group of order prime with $d!$, then $G$ is of type $\rR_{d}$. 

Since we know which $2$-dimensional singularities are of type $\rR$, then it is in principle possible to classify all groups of type $\rR_{2}$ in characteristic $p\ge 0$. For abelian groups, we prove that the only ones not of type $\rR_{2}$ are those of the form $C_{a}\times C_{2ab}$, where $C_{n}$ is the cyclic group of order $n$, see Proposition~\ref{r2-ab}. For non-abelian groups, we have a complete classification in characteristic $0$. The classification is fairly complex, it is summarized in a table at the end of the paper.

\subsection*{Acknowledgements}

This paper was born as part of my recent joint articles with A. Vistoli \cite{giulio-angelo-valuative} and \cite{giulio-angelo-moduli}. I am grateful to him for many useful discussions, and for an alternative proof of Proposition~\ref{prop:cyclic} in characteristic $0$, see Remark~\ref{rem:cyclic}. I thank J.-L. Colliot-Thélène for pointing out to me the concept of central points in real algebraic geometry, J. Koll\'ar for his interest in my work, C. Liedtke for suggesting how to generalize an argument to positive characteristic (see Remark~\ref{rem:cyclic}) and an anonymous referee for many useful comments.

\section{\texorpdfstring{$\rR$}{R}-singularities in dimension 2}

In this section we prove Theorem~\ref{thm:R2sing}. We refer to \cite[\S 6]{giulio-angelo-moduli} for the definitions of the fundamental gerbe and associated variety of a singularity $(S,s)$ \cite[Definition 6.2, Definition 6.8]{giulio-angelo-moduli}. The fundamental group of $(S,s)$ is the automorphism group of any geometric point of the fundamental gerbe of $(S,s)$, and it coincides with the local fundamental group of $S_{K}$ in $s_{K}$. 

\begin{proposition}\label{prop:cyclic}
	Let $(S,s)$ be a tame quotient singularity in dimension $2$ over $k$. If $(S,s)$ is not liftable, then its fundamental group is cyclic.
\end{proposition}

\begin{proof}
	Let $G$ be the fundamental group of $(S,s)$, we want to prove that $G$ is cyclic. Let $\bar{G}$ be the quotient of $G$ by its center, it is enough to prove that $\bar{G}$ is cyclic. In fact, if $\bar{G}$ is cyclic then $G$ is abelian, and since by \cite[Corollary 6.4]{giulio-angelo-moduli} there is an embedding $G\subset \GL_{2}(\bar{k})$ with no pseudoreflections then $G$ has rank $1$. So we want to show that $\bar{G}$ is cyclic.
	
	Let $\cG$ be the fundamental gerbe of the singularity, consider the inertia stack $\rI_{\cG}\to\cG$, it is a relative group over $\cG$. Let $\cZ\subset\rI_{\cG}$ be the center and define $\bar{\cG}$ as the rigidification $\rI_{\cG}\thickslash\cZ$ (see \cite[Appendix C]{dan-tom-angelo2008}), we have that $\bar{\cG}$ is a quotient gerbe of $\cG$ and the automorphism group of any geometric point of $\bar{\cG}$ is isomorphic to $\bar{G}$.
	
	 Let $E$ be the associated variety of the singularity, it is a smooth projective curve of genus $0$ and there is a rational map $E\dashrightarrow\cG$ \cite[Corollary 6.9]{giulio-angelo-moduli}. I claim that the geometric fibers of the composite $E\dashrightarrow\cG\to\bar{\cG}$ are connected curves of genus $0$. In fact, let $\bar{G}'$ be the image of $G\subset\GL_{2}(\bar{k})$ in $\PGL_{2}(\bar{k})$, then there is a factorization $G\to\bar{G}'\to\bar{G}$ and the composite $E_{\bar{k}}\simeq\PP^{1}_{\bar{k}}/\bar{G}'\dashrightarrow \cG_{\bar{k}}=\cB_{\bar{k}}G\to\cB_{\bar{k}}\bar{G}'$ corresponds to the natural covering $\PP^{1}_{\bar{k}}\to\PP^{1}_{\bar{k}}/\bar{G}'$, see \cite[\S 6.4]{giulio-angelo-moduli}. In particular, the fibers of $E_{\bar{k}}\simeq\PP^{1}_{\bar{k}}/\bar{G}'\dashrightarrow\bar{\cG}$ are dominated by $\PP^{1}_{\bar{k}}$, hence they have genus $0$.
	
	If $(S,s)$ is not liftable, then $E$ has no rational points \cite[Proposition 6.5, Proposition 6.10]{giulio-angelo-moduli}. By \cite[Proposition 2]{giulio-divisor} applied to $E\dashrightarrow\bar{\cG}$ we get that $\bar{G}$ is cyclic.
\end{proof}

\begin{remark}\label{rem:cyclic}
	The first proof of Proposition~\ref{prop:cyclic} was found by A. Vistoli, let us sketch it. If $(S,s)$ is a singularity, by descent theory the minimal resolution of $(S_{\bar{k}},s_{\bar{k}})$ descends to $k$.
	
	If $(S_{\bar{k}},s_{\bar{k}})$ is not cyclic and the characteristic is $0$, then by Brieskorn's classification \cite{brieskorn} the minimal resolution has a unique component which is a rational curve intersecting three other components. C. Liedtke pointed us to the fact that this is still true in positive characteristic: in fact, one can reduce to characteristic $0$ using \cite[Theorem 4.10]{liedtke}. Because of this, said component descends to a geometrically irreducible component of the minimal resolution of $S$ which is thus a regular, projective curve of genus $0$ with a divisor of degree $3$, i.e. $\PP^{1}$.
\end{remark}

Thanks to Proposition~\ref{prop:cyclic}, we can focus on cyclic singularities. Let $1\le d\le n$ be integers, assume that $n$ is prime with $d$ and $\cha k$.

It is well-known that the singularity $\frac 1n(1,d)$ over a separably closed field has a minimal resolution whose exceptional divisor is a string of rational curves with autointersections $(-a_{0},\dots,-a_{r})$, where the integers $a_{0},\dots,a_{r}$ are given by the Hirzebruch-Jung expansion

\[
\frac{n}{d} = a_{0} - \cfrac{1}{a_{1} - \cfrac{1}{a_{2} - \cfrac{1}{\cdots}
}}
\]

For historical reasons, references usually work over $\CC$ e.g. \cite{reid}. Still, it is clear that the argument carries over to any separably closed base field. 

\begin{proposition}\label{prop:cycR}
	Let $n,d$ be as above, and let $(a_{0},\dots,a_{r})$ be the Hirzebruch-Jung expansion of $n/d$. Let $p\nmid n$ be either a prime or $p=0$. If $\frac 1n(1,d)$ is not of type $\rR$ in characteristic $p$, then $a_{i}=a_{r-i}$ for every $i$, $r$ is even and $a_{r/2}$ is even.
\end{proposition}

\begin{proof}
	Let $(S,s)$ be a non-liftable singularity of type $\frac 1n(1,d)$ over a field $k$ with separable closure $K$, and let $(a_{0},\dots,a_{r})$ be the Hirzebruch-Jung expansion of $n/d$. The minimal resolution of the pullback to $K$ descends to $k$, denote it by $\tilde{S}\to S$ and let $E\subset\tilde{S}$ be the exceptional divisor, then $E_{K}$ is a string of $r+1$ rational curves with autointersections $-a_{0},\dots,-a_{r}$. Since $(S,s)$ is not liftable, then $E(k)=\emptyset$.
	
	If $a_{i}\neq a_{r-i}$ for some $i$, then the Galois action does not permute the components of $E_{K}$ and $r\neq 0$, hence any intersection point of two components of $E_{K}$ descends to a rational point of $E$, which is absurd. Assume that $a_{i}=a_{r-i}$ for every $i$. If $r$ is odd, then $E_{K}$ has a central point which is Galois invariant, hence $E(k)\neq\emptyset$. If $r$ is even and $a_{r/2}$ is odd, then the central component of $E_{K}$ descends to a genus $0$ curve $C\subset E$ with odd autointersection. In particular, $C$ has a divisor of odd degree, hence $C\simeq\PP^{1}$ and $E(k)\neq\emptyset$.
	
\end{proof}

\begin{definition}
	A pair of positive integers $(n,d)$ with $d<n$ and $\gcd(n,d)=1$ is a \emph{critical pair} if it satisfies the conclusion of Proposition~\ref{prop:cycR}: namely, the Hirzebruch-Jung expansion of $n/d$ is symmetric, of odd length and with even central term.
\end{definition}

Proposition~\ref{prop:cycR} tells us that we need to search for $\neg\rR$ singularities only among critical pairs. We are going to give a simple characterization of critical pairs, and then we'll show that a singularity of type $\frac1n(1,d)$ with $(n,d)$ critical is not of type $\rR$ in every characteristic $p\ge 0$, $p\nmid n$.

\subsection{Characterization of critical pairs}

In the following, $d,n$ will always be positive integers with $1\le d\le n$, $\gcd(n,d)=1$. All the results of this subsection are numerical facts about critical pairs. We only use geometry as a tool for computing properties of Hirzebruch-Jung expansions, hence we may work over $\CC$ for simplicity.

\begin{lemma}\label{lem:sym}
	The Hirzebruch-Jung expansion of $n/d$ is symmetric if and only if $d^{2}\cong 1\pmod{n}$.
\end{lemma}

\begin{proof}
	Let $1\le b\le n$ be such that $bd\cong 1\pmod{n}$, then $d^{2}\cong 1\pmod{n}$ if and only if $b=d$. Let $(a_{0},\dots,a_{r})$ be the expansion of $n/d$. The lattice, and hence the Newton polygon, of $\frac1n(1,b)$ cf. \cite[\S 2]{reid} is obtained by the one of $\frac1n(1,d)$ by reflecting along the line $x=y$, hence the Hirzebruch-Jung expansion of $n/b$ is given by $(a_{r},\dots,a_{0})$. Alternatively, $\zeta_{n}'=\zeta_{n}^{d}$ is a primitive $n$-th root of $1$ satisfying $\zeta_{n}'^{b}=\zeta_{n}$, hence the singularity $\frac1n(1,b)$ is obtained from $\frac1n(1,d)$ by swapping the eigenvectors of the representation and its Hirzebruch-Jung expansion is therefore $(a_{r},\dots,a_{0})$. The statement follows.  
\end{proof}

Let $X_{n,d}$ be the quotient of $\AA^{2}$ by the action of $\mu_{n}$ with eigenvalues $\zeta_{n},\zeta_{n}^{d}$. From now on, assume $d^{2}\cong 1\pmod{n}$. If the expansion of $n/d$ has odd length, denote by $\tilde{X}_{n,d}$ the minimal resolution of singularities, otherwise $\tilde{X}_{n,d}$ is the blow up of the minimal resolution in the central point. In any case, the exceptional divisor of $\tilde{X}_{n,d}$ is a string of rational curves of odd length. Observe that $(n,d)$ is a critical pair if and only if the central component of the exceptional divisor of $\tilde{X}_{n,d}$ has even autointersection.

Since $d^{2}\cong 1\pmod{n}$, it is immediate to check that the linear involution of $\AA^{2}$ which swaps the coordinates descends to an involution of $X_{n,d}$, and in turn this induces an involution of $\tilde{X}_{n,d}$. Denote by $C_{2}$ the group of order $2$, we have a $C_{2}$ action on all of these surfaces.

Assume that $n=n'\cdot m$, let $1\le d'\le n'$ be such that $d'\cong d\pmod{n'}$, then we have a natural $C_{2}$-equivariant morphism $X_{n',d'}\to X_{n,d}$ and hence a $C_{2}$-equivariant rational map $\tilde{X}_{n',d'}\dashrightarrow \tilde{X}_{n,d}$.

\begin{lemma}\label{lem:res}
	Let $n,d,n',d'$ be as above. There exists a regular surface $Y$ with an involution and $C_{2}$-equivariant proper morphisms $Y\to \tilde{X}_{n',d'}, Y\to\tilde{X}_{n,d}$ such that the diagram
	\[\begin{tikzcd}[column sep=small]
								&	Y\ar[dl]\ar[dr]	&				\\
		\tilde{X}_{n',d'}\ar[rr,dashed]	&					&	\tilde{X}_{n,d}
	\end{tikzcd}\]
	commutes and $Y\to X_{n',d'}$ is birational. Furthermore, the exceptional divisor of $Y$ has an irreducible component $E\subset Y$ with three properties: first, the involution maps $E$ to itself, second, $E$ is the unique component dominating the central component of the exceptional divisor of $\tilde{X}_{n,d}$, third, the same but for $\tilde{X}_{n',d'}$.
\end{lemma}

\begin{proof}
	Let $B$ be the blowup of $\AA^{2}$ in the origin and $Z_{n,d}$ the quotient of $B$ by the induced actions of $\mu_{n}$. Using the Chevalley-Shephard-Todd theorem, it is easy to check that the surface $Z_{n,d}$ has either $0$ or $2$ singular points, and if they exist they are swapped by the involution. Let $Z_{n,d}^{\circ}$ be the smooth locus. Since $\tilde{X}_{n,d}$ is the minimal resolution of $X_{n,d}$ plus (possibly) a blow up, the universal properties of the minimal resolution and of the blow up give us a morphism $Z_{n,d}^{\circ}\to\tilde{X}_{n,d}$ which is readily checked to be an open embedding. We can obtain analogously an open embedding $Z_{n',d'}^{\circ}\subset\tilde{X}_{n',d'}$. The point of this construction is that there is a morphism $Z_{n',d'}\to Z_{n,d}$, and hence the rational map $\tilde{X}_{n',d'}\to\tilde{X}_{n,d}$ is defined on the open subset $Z_{n',d'}^{\circ}$.
	
	Let $S\subset \tilde{X}_{n',d'}$ be the finite set of points where $\tilde{X}_{n',d'}\dashrightarrow\tilde{Z}_{n,d}$ is not defined, by symmetry it is preserved by the involution.	Thanks to the above, $S$ maps to the singular locus of $Z_{n',d'}$, hence the action of $C_{2}$ on $S$ has no fixed points. Write $S=S_{1}\cup S_{2}$, with $S_{i}$ containing one point for each orbit. Let $Y_{1}\to \tilde{Z}_{n',d'}\setminus S_{2}$ be a resolution of $\tilde{Z}_{n',d'}\setminus S_{2}\to\tilde{Z}_{n,d}$, and define $Y_{2}$ as the composite $Y_{2}=Y_{1}\to \tilde{Z}_{n',d'}\setminus S_{2}\to\tilde{Z}_{n',d'}\setminus S_{1}$ where the second map is the involution. We may then obtain $Y$ by gluing $Y_{1}$ and $Y_{2}$.
\end{proof}

\begin{corollary}\label{cor:critical-odd}
	Let $n,n',m,d,d'$ be as above, and assume that $m$ is odd. Then $(n,d)$ is a critical pair if and only if $(n',d')$ is a critical pair.
\end{corollary}

\begin{proof}
	Let $Y$ be as in Lemma~\ref{lem:res} and let $E\subset Y, F\subset \tilde{X}_{n,d}, F'\subset X_{n',d'}$ be the central components, we want to show that $F^{2}\cong F'^{2}\pmod{2}$. Since $E\subset Y$ is the only curve dominating $F$ and $Y\to\tilde{X}_{n,d}$ has odd degree, the degree of $E\to F$ and the ramification of $Y\to X_{n,d}$ in $E$ are both odd. Using the involution and projection formula, this implies that $E^{2}\cong F^{2}\pmod{2}$. Similarly, $E^{2}\cong F'^{2}\pmod{2}$, hence we conclude.
\end{proof}

\begin{corollary}\label{cor:critical}
	Let $1\le d\le n$ be integers with $\gcd(d,n)=1$, let $2^{a}$ be the largest power of $2$ which divides $n$. Then $(n,d)$ is a critical pair if and only if $d^{2}\cong 1\pmod{n}$, $a\ge 1$ and $d\cong\pm 1\pmod{2^{a}}$. 
\end{corollary}

\begin{proof}
	By Corollary~\ref{cor:critical-odd} we may reduce to the case $n=2^{a}$. If $a=0$, then clearly $(1,1)$ is not a critical pair. There are at most four values of $d$ modulo $2^{a}$ such that $d^{2}\cong 1\pmod{2^{a}}$, and these are $1,2^{a-1}-1,2^{a-1}+1,2^{a}-1$. The statement then follows by computing explicitly the Hirzebruch-Jung expansions of $2^{a}/d$ for these four values of $d$. Here are the expansions.
	\begin{description}
			\item[$d=1$:] $(2^{a})$, odd length $1$, even central term.
			\item[$d=2^{a-1}-1$:] if $a\ge 3$, $(3,2,2,\dots,2,2,3)$, even length $2^{a-1}-2$.
			\item[$d=2^{a-1}+1$:] if $a\ge 3$, $(2,2^{a-2}+1,2)$, odd length $3$, odd central term.
			\item[$d=2^{a}-1$:] $(2,\dots,2)$, odd length $2^{a}-1$, even central term.
	\end{description} 
\end{proof}

We now go back to working over a field $k$ of arbitrary characteristic.

\begin{corollary}\label{cor:ch2}
	Every $2$-dimensional tame quotient singularity over a field of characteristic $2$ is liftable.
\end{corollary}

\begin{proof}
	Let $(S,s)$ be a non-liftable tame quotient singularity over a field $k$ of characteristic $2$, by Propositions~\ref{prop:cyclic} and \ref{prop:cycR} its base change to $\bar{k}$ is of type $\frac 1n(1,d)$ with $n$ odd and $(n,d)$ critical. This is in contradiction with Corollary~\ref{cor:critical}. 
\end{proof}

Because of Corollary~\ref{cor:ch2}, from now on we assume $\cha k\neq 2$. 

\subsection{Non-liftable singularities}

Since we are assuming $\cha k\neq 2$, the $2$-torsion subgroup $\br(k)[2]$ of the Brauer group corresponds to smooth conics. If $\br(k)[2]$ is trivial, then every $2$-dimensional tame quotient singularity is liftable since its associated variety cf. \cite[Definition 6.8]{giulio-angelo-moduli} is isomorphic to $\PP^{1}$ \cite[Proposition 6.10]{giulio-angelo-moduli}. Assume now that $\br(k)[2]$ is non-trivial, and let $(n,d)$ be a critical pair with $\cha k\nmid n$. We are going to construct a non-liftable singularity of type $\frac1n(1,d)$. 

If $\alpha\in\br(k)[2]$ is non trivial, there exist $s,t\in k^{*}$ such that $s$ is not a square and $t$ is not a norm for the extension $k(s^{1/2})/k$, see for instance \cite[Proposition 1.1.7]{gille-szamuely}. 

Consider the lattice $L\subset\ZZ^{2}$ generated by $(n,0),(0,n),(n-d,1)$. Denote by $L^{+}$ the lattice points with non-negative coordinates, and by $\bfx_{L}$ the set of monomials $x^{\alpha}y^{\beta}$ with $(\alpha,\beta)\in L^{+}$. Then $\spec K[\bfx_{L}]$ is a singularity of type $\frac 1n(1,d)$, see \cite[\S 2]{reid}.

Let $o$ be the order of $d-1$ modulo $n$. Since $(n,d)$ is a critical pair, then $d\cong\pm 1$ modulo every prime power dividing $n$. Using this, it is easy to check that $2o$ divides both $n$ and $d+1$.

Let $s^{1/2}\in K$ be a square root of $s$, $t^{1/2o}\in K$ a $2o$-th root. Consider the involution $\sigma$ of $k(s^{1/2},t^{1/2o})[x,y]$ defined by
\[s^{1/2}\mapsto -s^{1/2}, ~ t^{1/2o}\mapsto t^{1/2o},\]
\[x\mapsto t^{1/2o}y, ~ y\mapsto t^{-1/2o}x,\]
we have
\[\sigma(x^{\alpha}y^{\beta})=t^{(\alpha-\beta)/2o}x^{\beta}y^{\alpha}.\]
If $(\alpha,\beta)\in L$, then $(\alpha-\beta)/2o$ is an integer since $2o$ divides both $n$ and $d+1$, and $L$ is generated by $(n,0),(0,n),(n-d,1)$. Furthermore, since $(n-d)^{2}\cong 1\pmod{n}$, if $(\alpha,\beta)\in L$ then $(\beta,\alpha)\in L$. It follows that $\sigma$ restricts to an involution of $k(s^{1/2})[\bfx_{L}]\subset k(s^{1/2},t^{1/2})[x,y]$. Let $A$ be the $k$-algebra of $\sigma$-invariant elements of $k(s^{1/2})[\bfx_{L}]$, then $\spec A$ is a singularity of type $\frac1n(1,d)$. Let us show that it is not liftable. 

Let $E$ be the associated variety of the singularity \cite[\S 6.4]{giulio-angelo-moduli}, it is enough to show that $E(k)$ is empty \cite[Proposition 6.10]{giulio-angelo-moduli}. It is easy to check that $E_{k(s^{1/2})}$ identifies naturally with $\proj k(s^{1/2})[x^{o},y^{o}]$. If $[\mu:\nu]\in \proj k(s^{1/2})[x^{o},y^{o}](k(s^{1/2}))=\PP^{1}(k(s^{1/2}))$ is a point, then $\sigma[\mu:\nu]=[t\sigma(\nu):\sigma(\mu)]$. Clearly, $[1:0]$ is not $\sigma$-invariant. If $[\mu:\nu]$ with $\nu\neq 0$ is $\sigma$-invariant, then we get that $t=\mu\nu^{-1}\sigma(\mu\nu^{-1})$ is a norm with respect to the extension $k(s^{1/2})/k$. This is in contradiction with our initial assumption, hence $E(k)$ is empty. We have thus proved the following.

\begin{lemma}\label{lem:non-lift}
	Let $(n,d)$ be a critical pair, and $k$ a field with $\cha k\nmid n$. The following are equivalent.
	\begin{itemize}
		\item The $2$-torsion subgroup $\br(k)[2]$ of the Brauer group is not trivial.
		\item There exists a non-liftable $2$-dimensional singularity $(S,s)$ of type $\frac 1n(1,d)$ over $k$.
	\end{itemize}
\end{lemma}

By putting together Proposition~\ref{prop:cyclic}, Proposition~\ref{prop:cycR}, Corollary~\ref{cor:critical} and Lemma~\ref{lem:non-lift}, we obtain a proof of Theorem~\ref{thm:R2sing}.

\begin{remark}
	Another way of constructing singularities with particular arithmetic properties, e.g. being non-liftable, is to use a theorem of J. Koll\'ar \cite[Corollary 5]{kollar} which classifies twisted forms of a singularity in terms of Galois cohomology. While for cyclic singularities in dimension $2$ we can simply use explicit constructions as in Lemma~\ref{lem:non-lift}, for higher dimensional singularities Koll\'ar's more conceptual approach might be helpful.
\end{remark}

\section{\texorpdfstring{$\rR_{2}$}{R2}-groups in characteristic \texorpdfstring{$0$}{0}}

Once we have a complete classification of $\rR$-singularities in dimension $2$, it is in theory possible to classify completely the groups of type $\rR_{2}$ by checking every finite subgroup of $\GL_{2}$. We do this in every characteristic for abelian groups, and in characteristic $0$ for non-abelian ones. Since most finite groups are $\rR_{2}$, our point of view will be that of searching for groups \emph{not} of type $\rR_{2}$, or $\neg\rR_{2}$.

\begin{proposition}\label{r2-ab}
Let $G$ be a finite abelian group and $p$ either $0$ or a prime not dividing $|G|$. Then $G$ is $\neg\rR_{2}$ in characteristic $p$ if and only if $G\simeq C_{a}\times C_{2ab}$ for some $a,b\ge 1$.
\end{proposition}

\begin{proof}
	The representation $(x,y)\mapsto \begin{psmallmatrix} \zeta_{2ab}^{2bx+y} & 0 \\ 0 & \zeta_{2ab}^{y} \end{psmallmatrix}$ of $C_{a}\times C_{2ab}$ generates a singularity of type $\frac{1}{2b}(1,1)$, hence $C_{a}\times C_{2ab}$ is $\neg\rR_{2}$ by Theorem~\ref{thm:R2sing}.
	
	On the other hand, assume that $G$ has a faithful representation $G\subset \GL_{2}(K)$ generating a singularity which is not of type $\rR$. Since $G$ is abelian, up to conjugation we may assume that the matrices of $G$ are diagonal. We may write $G\simeq C_{a}\times C_{ab}$ for some $a,b\ge 1$, we want to show that $b$ is even. By cardinality reasons, $C_{a}\times C_{a}\subseteq G\subset \GL_{2}(K)$ is the set of all diagonal matrices whose order is finite and divides $a$. Because of this, $C_{a}\times C_{a}$ is contained in the subgroup $P$ of $G$ generated by pseudoreflections. Since the associated singularity is not of type $\rR$, the index of $P$ is even by Theorem~\ref{thm:R2sing}. The index of $P$ divides $b$, which is thus even, as desired. 
\end{proof}

We now study non-abelian groups of type $\rR_{2}$ in characteristic $0$. We use the classification of finite subgroups of $\GL_{2}(\CC)$, which essentially goes back to Klein, and of which Cohen \cite{cohen} gives a modern account. Furthermore, we use Shephard and Todd's \cite{shephard-todd} classification of finite subgroups of $\GL_{2}(\CC)$ generated by pseudoreflections. Again, Cohen \cite{cohen} gives a modern account. If $K$ is an algebraically closed field of characteristic $0$, then every finite subgroup of $\GL_{2}(K)$ is contained in $\GL_{2}(\overline{\QQ})\subset\GL_{2}(K)$, hence we may use these classification over $K$ too. 

\begin{remark}
	While writing our joint paper with A. Vistoli \cite{giulio-angelo-moduli}, it was not clear to us whether a product of $\rR_{d}$ groups was itself $\rR_{d}$. From the classification of $\rR_{2}$ groups it will follow that, if $S_{4}'$ is the inverse image of $S_{4}\subset\PGL_{2}(K)$ in $\SL_{2}(K)$, then $C_{3}$ and $S_{4}'$ are of type $\rR_{2}$ in characteristic $0$, but their product $C_{3}\times S_{4}'$ is not.
\end{remark}

\subsection{Set-up and notations}\label{sec:notation}

We will mostly follow the notations of \cite{cohen}. We write $\zeta_{n}=e^{2\pi i/n}\in\bar{\QQ}\subset K$ and denote by $\mu_{n}$ the cyclic subgroup of $\GL_{2}(K)$ generated by $\zeta_{n}\operatorname{Id}$.

Assume that $H$ is a finite subgroup of $\GL_{2}(K)$ and $K\subseteq H$ is a normal subgroup with cyclic quotient, and suppose that an isomorphism $\phi:\mu_{wd}/\mu_{d}\xrightarrow{\sim} H/K$ is given. Following Cohen, we define
\[(\mu_{wd} \mid \mu_{d} ; H \mid K )_{\phi}=\{\lambda h \mid \lambda\in\mu_{wd}, h\in H, \phi([\lambda])=[h]\in H/K\}\subset\GL_{2}(K)\]

In practice, every time we will use this notation the conjugacy class of the group will not depend on $\phi $, see \cite[p.393]{cohen}, hence we will drop the subscript. 

Cohen uses some subgroups of $\SL_{2}(K)$ to which he attaches names that we find confusing, such as $\bD_{n}$ for a group which is not dihedral. We use the same subgroups, but with different names. Our convention is that, for $G=C_{n},D_{n},A_{4},S_{4},A_{5}$, we write $G'$ for a specific subgroup of $\SL_{2}(K)$ whose image in $\PGL_{2}(K)$ is isomorphic to $G$, with the exception that, for $n$ even, $C_{n}'$ has image isomorphic to $C_{n/2}$. We write Cohen's names between parentheses.

\[({\bC_{n}})~~C_{n}'=\left <\begin{pmatrix} \zeta_{n} & 0 \\ 0 & \zeta_{n}^{-1} \end{pmatrix}\right>;~~~ ({\bD_{n}})~~ D_{n}'=\left<\begin{pmatrix} 0 & i \\ i & 0 \end{pmatrix}, C_{2n}'\right>;\]
\[({\bT})~~ A_{4}'=\left<\frac{\zeta_{8}}{\sqrt{2}}\begin{pmatrix} 1 & i \\ 1 & -i \end{pmatrix}, D_{2}'\right>;~~~ ({\bO})~~ S_{4}'=\left<\begin{pmatrix} \zeta_{8} & 0 \\ 0 & \zeta_{8}^{-1} \end{pmatrix}, A_{4}'\right>.\]
We define $A_{5}'$ (Cohen's ${\bI})$ as any subgroup of $\SL_{2}(K)$ containing $-\operatorname{Id}$ whose image in $\PGL_{2}(K)$ is isomorphic to $A_{5}$: there is only one such group up to conjugation, and we are not interested in the specific representation.

Moreover, given integers $p|m$, the group $G(m,p,2)$ is defined as the subgroup of $\GL_{2}(K)$ with elements
\[G(m,p,2)=\left\{\begin{psmallmatrix} 0 & \zeta_{m}^{a} \\ \zeta_{m}^{b} & 0 \end{psmallmatrix} ,\begin{psmallmatrix} \zeta_{m}^{a} & 0 \\ 0 & \zeta_{m}^{b} \end{psmallmatrix} \mid a+b\cong 0\pmod{p} \right\}.\]
The group $G(m,p,2)$ has order $2m^{2}/p$ and is generated by the pseudoreflections $\begin{psmallmatrix} 0 & 1 \\ 1 & 0 \end{psmallmatrix}, \begin{psmallmatrix} 0 & \zeta_{m} \\ \zeta_{m}^{-1} & 0 \end{psmallmatrix},\begin{psmallmatrix} 1 & 0 \\ 0 & \zeta_{m}^{p} \end{psmallmatrix}$, its image in $\PGL_{2}(K)$ is a dihedral group.

A finite subgroup of $\GL_{2}(K)$ generated by pseudoreflections is either abelian, conjugate to $G(m,p,2)$ for some $m,p$, or conjugate to one of the groups listed in \cite[Table p.395]{cohen}. The groups in this list are numbered from 3 to 22, we call them $\ST_{3},\ST_{4},\dots,\ST_{22}$ accordingly.

In the following $G\subset\GL_{2}(K)$ is a finite subgroup, $P\subset G$ the subgroup generated by pseudoreflections and $\bar{G},\bar{P}$ the images in $\PGL_{2}(K)$. We organize our search of $\neg\rR$ singularities $\AA^{2}/G$ by fixing the isomorphism class of $\bar{G}$. Up to conjugation, non-abelian finite subgroups of $\GL_{2}(K)$ are classified in \cite[Theorem p.393]{cohen} by nine infinite families: four families for $\bar{G}=D_{m}$, two for $A_{4}$, two for $S_{4}$ and one for $A_{5}$. The families corresponding to $D_{m}$ are parametrized by two positive integers $m,q$, the others by one positive integer $q$ (Cohen parametrizes with $m$ these last families, but using $q$ is more coherent with the notation for the first families). We write $g=\gcd(m,q)$, $m'=m/g$, $q'=q/g$.

If $\bar{G}\simeq C_{m}$, then $G$ is abelian, hence we can ignore this case. For $\bar{G}=D_{m}$ we consider each of the corresponding four families of subgroups of $\GL_{2}(K)$ and compute directly the subgroup $P$. In the remaining cases such a direct computation would be much more complex since the elements of $G$ are not as easily written, so we use a different strategy. A necessary condition for $\AA^{2}/G$ to be $\neg\rR$ is that $\bar{G}/\bar{P}$ is cyclic: for $\bar{G}=A_{4},S_{4},A_{5}$ this puts strong constraints on $\bar{P}$ and hence, using the Shephard-Todd classification, on $P$.

Let $1\le d\le n$ be as usual and $1\le b\le n$ such that $bd\simeq 1\pmod{n}$. If $G\subset\GL_{2}(K)$ is a representation with dual $G^{\vee}=G\subset\GL_{2}(K)$ such that $\AA^{2}/G\simeq \frac1n(1,d)$, then $\AA^{2}/G^{\vee}\simeq\frac 1n(1,b)$ is the quotient of the dual representation. Using Theorem~\ref{thm:R2sing} it is easy to see that $\frac1n(1,d)$ is of type $\rR$ if and only if the same holds for $\frac1n(1,b)$. We sometimes work with the dual representation $G^{\vee}\subset\GL_{2}(K)$ because the computations naturally yield coordinates on the cotangent space of $\AA^{2}/P$, rather than coordinates of the tangent space.

\section{The case \texorpdfstring{$\bar{G}\simeq D_{m}$}{Dm}} 

\begin{proposition}\label{Dr2}
	Assume that $\bar{G}\simeq D_{m}$. The singularity $\AA^{2}/G$ is $\neg\rR$ if and only if $G$ is conjugate to one of the following groups. 
	\begin{itemize}
		\item $(\mu_{4q} \mid \mu_{2q} ; D_{m}' \mid C_{2m}')$ with $q$ even, $m|q$ and $q/m$ odd.
		\item $(\mu_{4q} \mid \mu_{2q} ; D_{2m}' \mid D_{m}')$ with $(2q',m'+q')$ a critical pair.
		\item $\mu_{2q}D_{m}'$ with $q$ odd and $m|q$.
		\item $\mu_{2q}D_{m}'$ with $q$ even and $(q',m')$ a critical pair.
	\end{itemize}
\end{proposition}

\subsection{\texorpdfstring{$\bm{G=(\mu_{4q} \mid \mu_{2q} ; D_{m}' \mid C_{2m}')}$}{Dm1}} We are going to show that $\AA^{2}/G$ is $\neg \rR$ if and only if $q$ is even, $m|q$ and $q/m$ is odd. Let us write explicitly the elements of $G$.
\[G=\left\{\zeta_{4q}^{2a+q+1}\begin{psmallmatrix} 0 & \zeta_{2m}^{-b} \\ \zeta_{2m}^{b} & 0 \end{psmallmatrix} ,\zeta_{2q}^{a}\begin{psmallmatrix} \zeta_{2m}^{b} & 0 \\ 0 & \zeta_{2m}^{-b} \end{psmallmatrix} \mid 0\le a < 2q, 0\le b < 2m \right\}\]
The order of $G$ is $4qm$ (in the presentation above, each element appears twice). We want to compute the subgroup $P\subseteq G$ generated by pseudoreflections. Let us divide by cases.

\subsubsection{If $q$ is odd} The pseudoreflection in $G$ are
\[\left\{\begin{psmallmatrix} 0 & \zeta_{2m}^{-b} \\ \zeta_{2m}^{b} & 0 \end{psmallmatrix} ,\begin{psmallmatrix} \zeta_{g}^{a} & 0 \\ 0 & 1 \end{psmallmatrix},\begin{psmallmatrix} 1 & 0 \\ 0 & \zeta_{g}^{a} \end{psmallmatrix} \mid 0 \le a < g, 0\le b < 2m \right\}.\]
In particular, it is easy to see that the generated group $P$ contains the subgroup $<\begin{psmallmatrix} 0 & 1 \\ 1 & 0 \end{psmallmatrix},C_{2m}'>$ whose order is $4m$, hence the index of $P$ divides $q$ and thus is odd. It follows that $\AA^{2}/G$ is of type $\rR$.

\subsubsection{If $q$ is even} The pseudoreflection in $G$ are
\[\left\{\begin{psmallmatrix} \zeta_{g}^{a} & 0 \\ 0 & 1 \end{psmallmatrix},\begin{psmallmatrix} 1 & 0 \\ 0 & \zeta_{g}^{a} \end{psmallmatrix} \mid 0 \le a < g \right\},\]
hence
\[P=\left\{\begin{psmallmatrix} \zeta_{g}^{a} & 0 \\ 0 & \zeta_{g}^{b} \end{psmallmatrix} \mid 0 \le a,b < g \right\}\]
has order $g^{2}$ and index $4q'm'$. If $\bar{G},\bar{P}$ are the images in $\PGL_{2}(K)$, then $\bar{G}/\bar{P}\simeq D_{m}/C_{g}\simeq D_{m/g}$. If $g\neq m$, this group is not cyclic, hence the same holds for $G/P$ and $\AA^{2}/G$ is of type $\rR$.

If $g=m$, i.e. $m|q$, the class $\omega\in G/P$ of $\zeta_{4q}\begin{psmallmatrix} 0 & 1 \\ 1 & 0 \end{psmallmatrix}$ has order $4q'=|G/P|$, hence $G/P$ is cyclic.  The ring of invariants $K[x,y]^{P}$ is $K[x^{m},y^{m}]$ and $\omega$ acts on the cotangent space of $\spec K[x^{m},y^{m}]$ with matrix $\begin{psmallmatrix} 0 & \zeta_{4q'} \\ \zeta_{4q'} & 0 \end{psmallmatrix}$. The eigenvalues of this matrix are $\zeta_{4q'},-\zeta_{4q'}=\zeta_{4q'}^{2q'+1}$, hence $\AA^{2}/G^{\vee}\simeq\frac{1}{4q'}(1,2q'+1)$. It is easy to check that $(4q',2q'+1)$ is a critical pair if and only if $q'$ is odd.

\subsection{\bm{$G=(\mu_{4q} \mid \mu_{2q} ; D_{2m}' \mid D_{m}')$}} We are going to prove that $\AA^{2}/G$ is $\neg\rR$ if and only if $(2q',m'+q')$ is a critical pair (in particular, $q'$ and $m'$ must have different parity).

Let us write explicitly the elements of $G$.
\[G=\left\{\zeta_{4q}^{2a+q+1}\begin{psmallmatrix} 0 & \zeta_{4m}^{-2b-1} \\ \zeta_{4m}^{2b+1} & 0 \end{psmallmatrix} ,\zeta_{4q}^{2a+1}\begin{psmallmatrix} \zeta_{4m}^{2b+1} & 0 \\ 0 & \zeta_{4m}^{-2b-1} \end{psmallmatrix},\right.\]
\[\left.\zeta_{4q}^{2a+q}\begin{psmallmatrix} 0 & \zeta_{2m}^{-b} \\ \zeta_{2m}^{b} & 0 \end{psmallmatrix} ,\zeta_{2q}^{a}\begin{psmallmatrix} \zeta_{2m}^{b} & 0 \\ 0 & \zeta_{2m}^{-b} \end{psmallmatrix}
\mid 0\le a < 2q, 0\le b < 2m \right\}\]
The order of $G$ is $8qm$ (in the presentation above, each element appears twice). We want to compute the subgroup $P\subseteq G$ generated by pseudoreflections. Let us divide by cases.

\subsubsection{If $q,m$ are odd} The pseudoreflection in $G$ are
\[\left\{\begin{psmallmatrix} 0 & \zeta_{4m}^{-2b-1} \\ \zeta_{4m}^{2b+1} & 0 \end{psmallmatrix} ,\begin{psmallmatrix} \zeta_{2g}^{a} & 0 \\ 0 & 1 \end{psmallmatrix},\begin{psmallmatrix} 1 & 0 \\ 0 & \zeta_{2g}^{a} \end{psmallmatrix} \mid 0 \le a < 2g, 0\le b < 2m \right\}.\]
The generated group $P$ contains a subgroup $\left\{\begin{psmallmatrix} 0 & \pm i \\ \pm i & 0 \end{psmallmatrix},\begin{psmallmatrix} \pm 1 & 0 \\ 0 & \pm 1 \end{psmallmatrix}\right\}$ of order $8$, hence $P$ has odd index and thus $\AA^{2}/G$ is of type $\rR$.

\subsubsection{If $v_{2}(q)=v_{2}(m) \ge 1$} The pseudoreflections in $G$ are
\[\left\{\begin{psmallmatrix} 0 & \zeta_{2m}^{-b} \\ \zeta_{2m}^{b} & 0 \end{psmallmatrix}, \begin{psmallmatrix} \zeta_{2g}^{a} & 0 \\ 0 & 1 \end{psmallmatrix},\begin{psmallmatrix} 1 & 0 \\ 0 & \zeta_{2g}^{a} \end{psmallmatrix} \mid 0 \le a < 2g, 0 \le b < 2m \right\},\]
hence $P=G(2m,m',2)$ has order $8gm$ and index $q'$, which is odd. It follows that $\AA^{2}/G$ is of type $\rR$.

\subsubsection{If $q$ is even, $v_{2}(q) \neq v_{2}(m)$} The pseudoreflections in $G$ are
\[\left\{\begin{psmallmatrix} 0 & \zeta_{2m}^{-b} \\ \zeta_{2m}^{b} & 0 \end{psmallmatrix}, \begin{psmallmatrix} \zeta_{g}^{a} & 0 \\ 0 & 1 \end{psmallmatrix},\begin{psmallmatrix} 1 & 0 \\ 0 & \zeta_{g}^{a} \end{psmallmatrix} \mid 0 \le a < g, 0 \le b < 2m \right\},\]
hence $P=G(2m,2m',2)$ has order $4gm$ and even index $2q'$. The class $\omega$ of $\zeta_{4q}\begin{psmallmatrix} \zeta_{4m} & 0 \\ 0 & \zeta_{4m}^{-1} \end{psmallmatrix}$ in the quotient $G/P$ has order equal to $2q'$, hence $G/P$ is cyclic generated by $\omega$. The ring of invariants is $K[x,y]^{P}=K[(xy)^{g},(x^{2m}+y^{2m})]$, see \cite[\S 2.5]{cohen}, hence $\omega$ acts on the cotangent space with matrix $\begin{psmallmatrix} \zeta_{2q'} & 0 \\ 0 & \zeta_{2q'}^{m'+q'} \end{psmallmatrix}$ and thus we have a singularity of type $\frac{1}{2q'}(1,m'+q')$, which is $\neg\rR$ if and only if $(2q',m'+q')$ is a critical pair.

\subsubsection{If $q$ is odd and $m$ is even} The pseudoreflection in $G$ are
\[\left\{\begin{psmallmatrix} 0 & \zeta_{4m}^{-2b-1} \\ \zeta_{4m}^{2b+1} & 0 \end{psmallmatrix} ,\begin{psmallmatrix} \zeta_{g}^{a} & 0 \\ 0 & 1 \end{psmallmatrix},\begin{psmallmatrix} 1 & 0 \\ 0 & \zeta_{g}^{a} \end{psmallmatrix} \mid 0 \le a < g, 0\le b < 2m \right\}.\]

Consider the matrix $A=\begin{psmallmatrix} \zeta_{8m}^{-1} & 0 \\ 0 & \zeta_{8m} \end{psmallmatrix}$, then we have $A^{-1}PA=G(2m,2m',2)$, which has order $4mg$ and index $2q'$. The class $\omega\in A^{-1}GA/A^{-1}PA$ of $\zeta_{4q}\begin{psmallmatrix} \zeta_{4m} & 0 \\ 0 & \zeta_{4m}^{-1} \end{psmallmatrix}$ has order $2q'$, hence $G/P$ is cyclic. The rest of this sub-case is analogous to the preceding sub-case, thus we get that $\AA^{2}/G$ is $\neg\rR$ if and only if $(2q',m'+q')$ is a critical pair.

\subsection{\bm{$G=\mu_{2q}D_{m}'$}}\label{muD} We are going to prove that $\AA^{2}/G$ is $\neg\rR$ if $q$ is odd and $m|q$ or if $q$ is even and $(q',m')$ is a critical pair.

Let us write explicitly the elements of $G$.
\[G=\left\{\zeta_{4q}^{2a+q}\begin{psmallmatrix} 0 & \zeta_{2m}^{-b} \\ \zeta_{2m}^{b} & 0 \end{psmallmatrix} ,\zeta_{2q}^{a} \begin{psmallmatrix} \zeta_{2m}^{b} & 0 \\ 0 & \zeta_{2m}^{-b} \end{psmallmatrix} \mid 0\le a < 2q, 0\le b < 2m \right\}\]
The order of $G$ is $4qm$.

\subsubsection{If $q$ is odd} The pseudoreflections in $G$ are
\[\left\{\begin{psmallmatrix} \zeta_{g}^{a} & 0 \\ 0 & 1 \end{psmallmatrix},\begin{psmallmatrix} 1 & 0 \\ 0 & \zeta_{g}^{a} \end{psmallmatrix} \mid 0 \le a < g\right\}\]
hence $P\simeq C_{g}^{2}$ has index $4q'm'$. If $g\neq m$, then $\bar{G}/\bar{P}\simeq D_{m}/C_{g}\simeq D_{m'}$ is not cyclic, hence $\AA^{2}/G$ is of type $\rR$.

If $g=m$, i.e. $m \mid q$ or $m'=1$, then it is easy to check that the class $\omega\in G/P$ of $\zeta_{4q}\begin{psmallmatrix} 0 & 1 \\ 1 & 0 \end{psmallmatrix}$ has order $4q'$, thus $G/P$ is cyclic generated by $\omega$. The ring of invariants is $K[x,y]^{P}=K[x^{m},y^{m}]$ and $\omega$ acts on the cotangent space with matrix $\begin{psmallmatrix} 0 & \zeta_{4q'} \\ \zeta_{4q'} & 0 \end{psmallmatrix}$, which has eigenvalues $\pm \zeta_{4q'}$. The singularity $\AA^{2}/G^{\vee}$ is thus of type $\frac{1}{4q'}(1,1+2q')$, which is $\neg\rR$ since $(4q',1+2q')$ is a critical pair for $q'$ odd.

\subsubsection{If $q$ is even} The pseudoreflection in $G$ are
\[\left\{\begin{psmallmatrix} 0 & \zeta_{2m}^{-b} \\ \zeta_{2m}^{b} & 0 \end{psmallmatrix} , \begin{psmallmatrix} \zeta_{g}^{a} & 0 \\ 0 & 1 \end{psmallmatrix},\begin{psmallmatrix} 1 & 0 \\ 0 & \zeta_{g}^{a} \end{psmallmatrix} \mid 0 \le a < g, 0\le b < 2m \right\}\]
hence $P=G(2m,2m',2)$ which has order $4mg$ and index $q'$. The class $\omega\in G/P$ of $\zeta_{2q}$ has order $q'$ and thus generates $G/P$. The ring of invariants is $K[x,y]^{P}=K[(xy)^{g},x^{2m}+y^{2m}]$, see \cite[\S 2.5]{cohen}, and $\omega$ acts on the cotangent space with matrix $\begin{psmallmatrix} \zeta_{q'} & 0 \\ 0 & \zeta_{q'}^{m'} \end{psmallmatrix}$. It follows that $\AA^{2}/G$ is a singularity of type $\frac{1}{q'}(1,m')$ which is $\neg\rR$ if and only if $(q',m')$ is a critical pair.

\subsection{\bm{$G=(\mu_{4q} \mid \mu_{q} ; D_{m}' \mid C_{m}')}$ with $m$ odd} We are going to show that $\AA^{2}/G$ is always of type $\rR$.

\[G=\left\{\zeta_{4q}^{4a+q+1}\begin{psmallmatrix} 0 & \zeta_{m}^{-b} \\ \zeta_{m}^{b} & 0 \end{psmallmatrix} , \zeta_{2q}^{2a+q+1} \begin{psmallmatrix} \zeta_{m}^{b} & 0 \\ 0 & \zeta_{m}^{-b} \end{psmallmatrix},\right.\]
\[\left.\zeta_{4q}^{4a+3q+3}\begin{psmallmatrix} 0 & \zeta_{m}^{-b} \\ \zeta_{m}^{b} & 0 \end{psmallmatrix} , \zeta_{q}^{a} \begin{psmallmatrix} \zeta_{m}^{b} & 0 \\ 0 & \zeta_{m}^{-b} \end{psmallmatrix} \mid 0\le a < q, 0\le b < m \right\}\]
The order of $G$ is $4qm$ if $q$ is even and $2qm$ if $q$ is odd.

\subsubsection{If $q\cong 3\pmod{4}$} We can write the elements of $G$ as
\[G=\left\{\zeta_{q}^{a}\begin{psmallmatrix} 0 & \zeta_{m}^{-b} \\ \zeta_{m}^{b} & 0 \end{psmallmatrix} , \zeta_{q}^{a} \begin{psmallmatrix} \zeta_{m}^{b} & 0 \\ 0 & \zeta_{m}^{-b} \end{psmallmatrix} \mid 0\le a < q, 0\le b < m \right\}\]
and the pseudoreflections are
\[\left\{\begin{psmallmatrix} 0 & \zeta_{m}^{-b} \\ \zeta_{m}^{b} & 0 \end{psmallmatrix} , \begin{psmallmatrix} \zeta_{g}^{a} & 0 \\ 0 & 1 \end{psmallmatrix},\begin{psmallmatrix} 1 & 0 \\ 0 & \zeta_{g}^{a} \end{psmallmatrix}  \mid 0\le a < g, 0\le b < m \right\}.\]

We have that $P=G(m,m',2)$ which has order $2mg$ and odd index $q'$. It follows that $\AA^{2}/G$ is of type $\rR$.

\subsubsection{If $q\cong 1\pmod{4}$} We can write the elements of $G$ as
\[G=\left\{-\zeta_{q}^{a}\begin{psmallmatrix} 0 & \zeta_{m}^{-b} \\ \zeta_{m}^{b} & 0 \end{psmallmatrix} , \zeta_{q}^{a} \begin{psmallmatrix} \zeta_{m}^{b} & 0 \\ 0 & \zeta_{m}^{-b} \end{psmallmatrix} \mid 0\le a < q, 0\le b < m \right\}\]
and the pseudoreflections are
\[\left\{-\begin{psmallmatrix} 0 & \zeta_{m}^{-b} \\ \zeta_{m}^{b} & 0 \end{psmallmatrix} , \begin{psmallmatrix} \zeta_{g}^{a} & 0 \\ 0 & 1 \end{psmallmatrix},\begin{psmallmatrix} 1 & 0 \\ 0 & \zeta_{g}^{a} \end{psmallmatrix}  \mid 0\le a < g, 0\le b < m \right\}.\]

If $A=\begin{psmallmatrix} 0 & i \\ -i & 0 \end{psmallmatrix}$, then we have that $A P A^{-1}=G(m,m',2)$ which has order $2mg$, hence $P$ has odd index $q'$. It follows that $\AA^{2}/G$ is of type $\rR$.

\subsubsection{If $q$ is even} We can write the elements of $G$ as
\[G=\left\{\zeta_{4q}^{2a+1}\begin{psmallmatrix} 0 & \zeta_{m}^{-b} \\ \zeta_{m}^{b} & 0 \end{psmallmatrix} , \zeta_{2q}^{a} \begin{psmallmatrix} \zeta_{m}^{b} & 0 \\ 0 & \zeta_{m}^{-b} \end{psmallmatrix} \mid 0\le a < 2q, 0\le b < m \right\}\]
The pseudoreflections are
\[\left\{\begin{psmallmatrix} \zeta_{g}^{a} & 0 \\ 0 & 1 \end{psmallmatrix},\begin{psmallmatrix} 1 & 0 \\ 0 & \zeta_{g}^{a} \end{psmallmatrix} \mid 0 \le a < g\right\},\]
hence $P$ has order $g^{2}$ and index $4q'm'$.

If $g\neq m$, then $G/P$ is not cyclic since $\bar{G}/\bar{P}\simeq D_{m}/C_{g}$ is not cyclic, hence $\AA^{2}/G$ is of type $\rR$.

If $g=m$, i.e. $m|q$, then the class $\omega\in G/P$ of $\zeta_{4q}\begin{psmallmatrix} 0 & 1 \\ 1 & 0 \end{psmallmatrix}$ has order $4q'=|G/P|$ and hence $G/P$ is cyclic. The ring of invariants is $K[x,y]^{P}=K[x^{m},y^{m}]$ and $\omega$ acts on the cotangent space with matrix $\begin{psmallmatrix} 0 & \zeta_{4q'} \\ \zeta_{4q'} & 0 \end{psmallmatrix}$, which has eigenvalues $\pm \zeta_{4q'}$. The singularity $\AA^{2}/G^{\vee}$ is thus of type $\frac{1}{4q'}(1,1+2q')$, which is of type $\rR$ since $q'$ is even and hence $(4q',1+2q')$ is not a critical pair.

\section{The case \texorpdfstring{$\bar{G}\simeq A_{4}$}{A4}}

\begin{proposition}
	If $\bar{G}\simeq A_{4}$, then $\AA^{2}/G$ is $\neg\rR$ if and only if $G$ is conjugate to one of the following groups.
	\begin{itemize}
		\item $(\mu_{6q} \mid \mu_{2q} ; A_{4}' \mid D_{2}')$ for $q=4,8$.
		\item $\mu_{2q}A_{4}'$ for $4|q$.
	\end{itemize}
\end{proposition}

A necessary condition for $\AA^{2}/G$ to be $\neg\rR$ is that $\bar{G}/\bar{P}$ is cyclic. This leaves us only two possibilities for $\bar{P}$: either $\bar{P}=\bar{G}$ or $\bar{P}\simeq D_{2}\subset A_{4}$, $\bar{G}/\bar{P}\simeq C_{3}$.

\subsection{\bm{$\bar{P}=\bar{G}\simeq A_{4}$}}\label{a4a4} In this case, we have $G=\mu_{2q}P$ where $\mu_{2q}$ is the intersection of $G$ with the center of $\GL_{2}(K)$ and, up to conjugation, we may assume that $P$ is $\ST_{n}$ for $n=4,\dots,7$. Let $a$ be the index of $P$ in $G$, then $P\cap\mu_{2q}=\mu_{2q/a}$ and $G/P\simeq \mu_{2q}/\mu_{2q/a}\simeq C_{a}$, the singularity $\AA^{2}/G$ is cyclic of type $\frac1{a}(1,d)$ for some integer $d$. Since $G/P$ is generated by the class of the matrix $\zeta_{2q/a}\begin{psmallmatrix} 1 & 0 \\ 0 & 1 \end{psmallmatrix}$, whether $(a,d)$ is a critical pair can be easily decided using the degrees of the monomials generating the ring of $P$-invariants: these can be found in the third column of \cite[Table p.395]{cohen}. 

\subsubsection{$P=\ST_{4}$ or $P=\ST_{5}$} In these cases, $q$ is odd, since otherwise $G=\mu_{2q}P$ contains a strictly larger group generated by pseudoreflections, namely $\ST_{6}$ or $\ST_{7}$. Since $P\cap\mu_{2}=\mu_{2}$ for these two groups, we get that $2q/a$ is even, $a$ is odd and hence $\AA^{2}/G$ is of type $\rR$.

\subsubsection{$P=\ST_{6}=(\mu_{12} \mid \mu_{4} ; A_{4}' \mid D_{2}')$} We have $\mu_{2q/a}=\mu_{4}$, $a=q/2$ and $\zeta_{q/2}$ acts on the cotangent space of $\AA^{2}/P$ with eigenvalues $\zeta_{q/2},\zeta_{q/2}^{3}$. We thus have a singularity of type $\frac{1}{q/2}(1,3)$ which is $\neg\rR$ if and only if $q=4,8$. 

\subsubsection{$P=\ST_{7}=\mu_{12}A_{4}'$} We have $\mu_{2q/a}=\mu_{12}$, $a=q/6$ and $\zeta_{q/6}$ acts with eigenvalues $\zeta_{q/6},\zeta_{q/6}$. We thus have a singularity of type $\frac{1}{q/6}(1,1)$ which is $\neg \rR$ if and only if $12|q$.

\subsection{\bm{$\bar{P}\simeq D_{2}\subset A_{4}$}} There are two classes of finite subgroups $G\subset\GL_{2}(K)$ such that $\bar{G}\simeq A_{4}$ \cite[p.393]{cohen}: up to conjugation, we may assume that $G$ is either $\mu_{2q}A_{4}'$ or $(\mu_{6q} \mid \mu_{2q} ; A_{4}' \mid D_{2}')$, they both have order $24q$. Since $\bar{P}\simeq D_{2}$, in both cases $P$ is contained in $\mu_{2q}D_{2}'$. We have already computed the subgroup of $\mu_{2q}D_{2}'$ generated by pseudoreflections in \ref{muD}. If $q$ is odd then $P$ is trivial, hence $\AA^{2}/G$ is of type $\rR$ since $G$ is not cyclic. For $q$ even, we have that the subgroup $P_{0}\subseteq P$ generated by pseudoreflections of $\mu_{2q}D_{2}'$ is $G(4,2,2)$, which has order $16$ and index $3q/2$. If $4\nmid q$ then $3q/2$ is odd and hence $\AA^{2}/G$ is of type $\rR$, let us assume $4|q$ and write $q'=q/4$, $P_{0}$ has index $6q'$. We are interested in the cases in which $P_{0}=P$ (since otherwise $\bar{P}\not\simeq D_{2}$).

\subsubsection{$G=\mu_{8q'}A_{4}'$} If $3|q'$ then $\ST_{5}=\mu_{6}A_{4}'\subseteq G$ is generated by pseudoreflections, hence $\bar{P}\not\simeq D_{2}$. Assume $3\nmid q'$. It can be checked that the class $\omega$ of $\frac{\zeta_{8q'}\zeta_{8}}{\sqrt{2}}\begin{psmallmatrix} 1 & i \\ 1 & -i \end{psmallmatrix}\in\mu_{8q'}A_{4}'$ has order $6q'$ in $G/P_{0}$, which is thus cyclic. The ring of invariants is $K[x,y]^{P_{0}}=K[(xy)^{2},x^{4}+y^{4}]$ and hence $\omega$ acts on the cotangent space with matrix
\[-\frac{\zeta_{2q'}}{4}\begin{pmatrix} 2 & 12 \\ -1 & 2 \end{pmatrix}\]

whose eigenvalues are $\zeta_{6q'}^{3+4q'},\zeta_{6q'}^{3+2q'}$. Since $3\nmid q'$, then $\gcd(6q',3+4q')=\gcd(6q',3+2q')=1$, hence $G/P_{0}$ acts with no non-trivial pseudoreflections and $P=P_{0}$. If $d$ is the inverse of $3+4q'$ modulo $6q'$, then the singularity is of type $\frac{1}{6q'}(1,d(3+2q'))$. Observe that $(3+4q')^{2}\cong(3+2q')^{2}\pmod{6q'}$ and that $3+4q'\cong 3+2q'\pmod{2q'}$. This implies that, if $3\nmid q'$, then $(6q',d(3+2q'))$ is a critical pair and hence $\AA^{2}/G$ is $\neg\rR$.

\subsubsection{$G=(\mu_{24q'} \mid \mu_{8q'} ; A_{4}' \mid D_{2}')$} If $3\nmid q'$ then $G$ contains $\ST_{4}=(\mu_{6} \mid \mu_{2} ; A_{4}' \mid D_{2}')$ which is generated by pseudoreflections, hence $\bar{P}\not\simeq D_{2}$. Assume $3|q'$. The class $\omega$ of $\frac{\zeta_{24q'}\zeta_{8}}{\sqrt{2}}\begin{psmallmatrix} 1 & i \\ 1 & -i \end{psmallmatrix}\in G$ has order $6q'$ in $G/P_{0}$, which is thus cyclic. The ring of invariants is $K[x,y]^{P_{0}}=K[(xy)^{2},x^{4}+y^{4}]$ and hence $\omega$ acts on the tangent space with matrix
\[-\frac{\zeta_{6q'}}{4}\begin{pmatrix} 2 & 12 \\ -1 & 2 \end{pmatrix}\]
whose eigenvalues are $\zeta_{6q'}^{1+4q'},\zeta_{6q'}^{1+2q'}$. Since $3|q'$, then $\gcd(6q',1+2q')=\gcd(6q',1+4q')=1$, hence $G/P_{0}$ acts with no non-trivial pseudoreflections and thus $P=P_{0}$. If $d$ is the inverse of $1+4q'$ modulo $6q'$, then the singularity is of type $\frac{1}{6q'}(1,d(1+2q'))$. Since $(1+4q')^{2}\not\cong (1+2q')^{2}\pmod{6q'}$, this is not a critical pair and hence $\AA^{2}/G$ is of type $\rR$.

\section{The case \texorpdfstring{$\bar{G}\simeq S_{4}$}{S4}}

\begin{proposition}
	If $\bar{G}\simeq S_{4}$, then $\AA^{2}/G$ is $\neg\rR$ if and only if $G$ is conjugate to one of the following groups.
	\begin{itemize}
		\item $\mu_{2q}S_{4}'$ for $q=3,8,9,16$ or $24|q$,
		\item $(\mu_{4q} \mid \mu_{2q} ; S_{4}' \mid A_{4}')$ for $12|q$, $24\nmid q$.
	\end{itemize}
\end{proposition}

A necessary condition for $\AA^{2}/G$ to be $\neg\rR$ is that $\bar{G}/\bar{P}$ is cyclic. This leaves us two possibilities for $\bar{P}$: either $\bar{P}=\bar{G}$ or $\bar{P}\simeq A_{4}\subset S_{4}$.

\subsection{\bm{$\bar{P}=\bar{G}$}} In this case, we have $G=\mu_{2q}P$ for some $q\ge 1$ and, up to conjugation, we have $P=\ST_{n}$ with $8\le n\le 15$. Let $a$ be the index of $P$ in $G$, then $P\cap\mu_{2q}=\mu_{2q/a}$ and $G/P\simeq \mu_{2q}/\mu_{2q/a}\simeq C_{a}$, the singularity $\AA^{2}/G$ is cyclic of type $\frac1{a}(1,d)$ for some integer $d$. Since $G/P$ is generated by the class of the matrix $\zeta_{2q/a}\begin{psmallmatrix} 1 & 0 \\ 0 & 1 \end{psmallmatrix}$, whether $(a,d)$ is a critical pair can be easily decided using the degrees of the monomials generating the ring of $P$-invariants: these can be found in the third column of \cite[Table p.395]{cohen}. 

\subsubsection{$P=\ST_{n}$, $n=8,10,12,13,14,15$} In these cases, $q$ is odd, since otherwise $G=\mu_{2q}P$ contains a strictly larger group generated by pseudoreflections, namely $\ST_{9}$ or $\ST_{11}$. Since $P\cap\mu_{2}=\mu_{2}$ in every case, we get that $2q/a$ is even, $a$ is odd and hence $\AA^{2}/G$ is of type $\rR$.

\subsubsection{$P=\ST_{9}=\mu_{8}S_{4}'$} We have $\mu_{2q}\cap P=\mu_{8}$, $a=q/4$ and the ring of invariants $K[x,y]^{P}$ is generated by two homogeneous polynomials of degrees $8,24$. Since $G/P$ is generated by the class of $\zeta_{2q}$, the singularity is of type $\frac{1}{a}(1,3)$, which is $\neg \rR$ if and only if $a=2,4$ i.e. $q=8,16$.

\subsubsection{$P=\ST_{11}=\mu_{24}S_{4}'$} We have $\mu_{2q}\cap P=\mu_{24}$, $a=q/12$ and the ring of invariants $K[x,y]^{P}$ is generated by two homogeneous polynomials of degrees $24,24$. Since $G/P$ is generated by the class of $\zeta_{2q}$, the singularity is of type $\frac{1}{q'}(1,1)$, which is $\neg \rR$ if and only if $2|q'$ i.e. $24|q$.

\subsection{\bm{$\bar{P}\simeq A_{4}$}}

\begin{lemma}
	The normalizer of $A_{4}$ in $\PGL_{2}(K)$ is $S_{4}$.
\end{lemma}

\begin{proof}
	We can identify $A_{4}$ with the subgroup of $\PGL_{2}(K)$ mapping the subset $S=\{0,1,\infty,\zeta_{6}\}\subset\PP^{1}(K)$ to itself. For every $s\in S$, there exists a non-trivial element of $A_{4}$ fixing it, hence the same is true for $g(s)$ for every $g$ in the normalizer. There is only a finite number of points of $\PP^{1}(K)$ fixed by a non-trivial element of $A_{4}$, the normalizer permutes them and hence is finite. Since $S_{4}$ normalizes $A_{4}$ and it is not contained in any strictly larger finite subgroup, we conclude. 
\end{proof}

Fix embeddings $A_{4}\subset S_{4}\subset \PGL_{2}(K)$ as the images of the preferred embeddings $A_{4}'\subset S_{4}'\subset\SL_{2}(K)$. Up to conjugation, we may assume that $G$ is equal either to $\mu_{2q}S_{4}'$ or $(\mu_{4q} \mid \mu_{2q} ; S_{4}' \mid A_{4}')$ \cite[Theorem p.393]{cohen} which both have order $48q$, in particular $\bar{P}=A_{4},\bar{G}=S_{4}\subset\PGL_{2}(K)$. We know that $P$ is conjugate to $\ST_{n}$ for some $4\le n\le 7$: let us show that in fact $P=\ST_{n}$.

Let $A\in\SL_{2}(K)$ be a matrix such that $A P A^{-1}=\ST_{n}$. The image of $\ST_{n}$ for $n=4,\dots,7$ is the preferred embedding $A_{4}\subset\PGL_{2}(K)$, in particular $A P A^{-1}$ and $P$ have the same image $A_{4}\subset\PGL_{2}(K)$. This means that $A$ maps to an element of the normalizer of $A_{4}$ in $\PGL_{2}(K)$, which is $S_{4}\subset\PGL_{2}(K)$, hence $A\in S_{4}'$. Since $G$ is either $\mu_{2q}S_{4}'$ or $(\mu_{4q} \mid \mu_{2q} ; S_{4}' \mid A_{4}')$, there exists some $\lambda\in K^{*}$ such that $\lambda A\in G$. But $P$ is normal in $G$: this implies that $A P A^{-1}=P$, hence $P=\ST_{n}$ for some $n=4, \dots, 7$. 

\subsubsection{$G=\mu_{2q}S_{4}'$} Since $\frac{\zeta_{3}\zeta_{8}}{\sqrt{2}} \begin{psmallmatrix} 1 & i \\ 1 & -i \end{psmallmatrix}\in\ST_{n}$ for every $n=4, \dots, 7$ and $\frac{\zeta_{8}}{\sqrt{2}}\begin{psmallmatrix} 1 & i \\ 1 & -i \end{psmallmatrix}\in S_{4}'$, then $\zeta_{3}\in G$ and hence $3|q$, write $q'=q/3$. Moreover, $q$ is odd, since otherwise $G$ contains $\mu_{12}S_{4}'=\ST_{15}$ which is generated by pseudoreflections, and this is in contradiction with $\bar{P}=A_{4}$. Since $2\nmid q$ then $n\neq 6,7$, because $\zeta_{4}\in\ST_{6},\ST_{7}$. Since $3|q$, then $\ST_{4}\subset\ST_{5}\subseteq P$, hence $P=\ST_{5}=\mu_{6}A_{4}'$ which has index $48q/72=2q'$.

It is easy to see that the class $\omega$ of $\zeta_{2q}\begin{psmallmatrix} \zeta_{8} & 0 \\ 0 & \zeta_{8}^{-1} \end{psmallmatrix}$ has order $2q'$ in $G/P$, which is thus cyclic of even order. The homogeneous polynomials $xy(x^{4}-y^{4}),(x^{4}+y^{4})(x^{4}+6x^{2}y^{2}+y^{4})(x^{4}-6x^{2}y^{2}+y^{4})$ are $P$-invariant and algebraically independent of degrees $6,12$, it follows that they generate $K[x,y]^{P}$ since $|P|=72=6\cdot 12$. Hence, $\omega$ acts with matrix $\begin{psmallmatrix} \zeta_{2q'}^{2+q'} & 0 \\ 0 & \zeta_{2q'}^{4+q'} \end{psmallmatrix}$ on the cotangent space of the quotient. Since $q'$ is odd, the associated singularity is $\neg\rR$ if and only if $(2+q')^{2}\cong(4+q')^{2}\pmod{2q'}$, or equivalently $q=3$ or $q=9$. Hence we get the groups $\mu_{6}S_{4}'$ and $\mu_{18}S_{4}'$.

\subsubsection{$G=(\mu_{4q} \mid \mu_{2q} ; S_{4}' \mid A_{4}')$} As in the preceding case, we can show that $3|q$. If $4\nmid q$ then $P$ contains either $\ST_{10}$ or $\ST_{14}$, this is in contradiction with $\bar{P}=A_{4}$. We may thus assume $12|q$ and write $q'=q/12$, $G=(\mu_{48q'} \mid \mu_{24q'} ; S_{4}' \mid A_{4}')$.

Since $\ST_{n}\subseteq\ST_{7}\subset G$ for every $n=4, \dots, 7$, then $P=\ST_{7}=\mu_{12}A_{4}'$, it has index $48q/144=4q'$. The class $\omega$ of $\zeta_{48q'}\begin{psmallmatrix} \zeta_{8} & 0 \\ 0 & \zeta_{8}^{-1} \end{psmallmatrix}$ has order $4q'$ in $G/P$, which is thus cyclic of even order. The homogeneous polynomials $x^{2}y^{2}(x^{4}-y^{4})^{2},(x^{4}+y^{4})(x^{4}+6x^{2}y^{2}+y^{4})(x^{4}-6x^{2}y^{2}+y^{4})$ are $P$-invariant and algebraically independent of degrees $12,12$, it follows that they generate $K[x,y]^{P}$ since $|P|=144=12\cdot 12$. Hence, $\omega$ acts with matrix $\begin{psmallmatrix} \zeta_{4q'} & 0 \\ 0 & \zeta_{4q'}^{1+2q'} \end{psmallmatrix}$ on the cotangent space of the quotient and thus generates a singularity of type $\frac{1}{4q'}(1,1+2q')$, which is $\neg\rR$ if and only if $q'$ is odd. Hence we get the groups $(\mu_{48q'} \mid \mu_{24q'} ; S_{4}' \mid A_{4}')$ for $q'$ odd.

\section{The case \texorpdfstring{$\bar{G}\simeq A_{5}$}{A5}}
\begin{proposition}
	If $\bar{G}\simeq A_{5}$, then $\AA^{2}/G$ is $\neg\rR$ if and only if $G$ is conjugate to $\mu_{2q}A_{5}'$ for $q=4,8,12,16,20,24,36,40,72$ or $60|q$.
\end{proposition}

A necessary condition for $\AA^{2}/G$ to be not of type $\rR$ is that $\bar{G}/\bar{P}$ is cyclic, hence $\bar{P}=\bar{G}$ since $A_{5}$ is simple. This implies that $G=\mu_{2q}P$ for some $q$ and hence $G/P$ is cyclic generated by $\zeta_{2q}$. Let $a$ be the index of $P$, we have $\mu_{2q}\cap P=\mu_{2q/a}$, $G/P\simeq C_{a}$ and $\AA^{2}/G$ is of type $\frac1a(1,d)$ for some $d$. Since $G/P$ is generated by $\zeta_{2q}$, whether $(a,d)$ is a critical pair can be easily decided using the degrees of the monomials generating the ring of $P$-invariants: these can be found in the third column of \cite[Table p.395]{cohen}. 

Up to conjugation, $P=\ST_{n}$ for $n=16, \dots, 22$. 

\subsection{$P=\ST_{n},~n=16,18,20$} In these cases $q$ is odd, since otherwise $G=\mu_{2q}\ST_{n}$ contains a strictly larger group generated by pseudoreflections, namely $\ST_{17}$, $\ST_{19}$ or $\ST_{21}$. Since $P\cap\mu_{2}=\mu_{2}$, we get that $2q/a$ is even, $a$ is odd and hence $\AA^{2}/G$ is of type $\rR$.

\subsection{$P=\ST_{17}=\mu_{20}A_{5}'$} We have $a=q/10$ and the singularity is of type $\frac{1}{a}(1,3)$, which is $\neg\rR$ if and only if $a=2,4$, $q=20,40$.

\subsection{$P=\ST_{19}=\mu_{60}A_{5}'$} We have $a=q/30$ and the singularity is of type $\frac{1}{a}(1,1)$, which is $\neg\rR$ if and only if $2|a$, $60|q$.

\subsection{$P=\ST_{21}=\mu_{12}A_{5}'$} We have $a=q/6$ and the singularity is of type $\frac{1}{a}(1,5)$, which is $\neg\rR$ if and only if $a=2,4,6,12$, $q=12,24,36,72$.

\subsection{$P=\ST_{22}=\mu_{4}A_{5}'$} We have $a=q/2$ and the singularity is of type $\frac{1}{a}(3,5)$. If it is of type $\neq \rR$, then $3^{2}\cong 5^{2}\pmod{a}$, i.e. $a\mid 16$. If $a\mid 16$, since $11$ is the inverse of $3$ modulo $16$ up to reparametrizing we have a singularity of type $\frac1a(1,5\cdot 11)=\frac1a(1,7)$ which is $\neg\rR$ if and only if $a=2,4,8$, $q=4,8,16$.

\begin{landscape}

\section{Non-abelian finite subgroups \texorpdfstring{$G\subset \GL_{2}(K)$}{GL2} such that \texorpdfstring{$\AA^{2}/G$}{the quotient} is not of type \texorpdfstring{$\rR$}{R}}

\renewcommand{\arraystretch}{1.5}

\begin{center}\begin{tabular}{|| c | c | c | c | c ||} 
	\hline
	Family												& 	$\neg\rR$ sub-family			&	Order		&	Order of center	& Group/center\\
	\hline\hline
	$(\mu_{4q} \mid \mu_{2q} ; D_{m}' \mid C_{2m}')$	&	$2|m|q$ and $2\nmid q'$			&	$4mq$		&	$2q$			&	$D_{m}$		\\
	\hline
	$(\mu_{4q} \mid \mu_{2q} ; D_{2m}' \mid D_{m}')$	&	$(2q',m'+q')$ critical 			&	$8mq$		&	$2q$			&	$D_{2m}$	\\
	\hline
	$\mu_{2q}D_{m}'$ 									&	$2\nmid q$ and $m|q$, or		&	$4mq$		&	$2q$			&	$D_{m}$		\\
														&	$2|q$ and $(q',m')$ critical	&				&					&				\\
	\hline
	$(\mu_{4q} \mid \mu_{q} ; D_{2m+1}'\mid C_{2m+1}')$ 	& 	$\emptyset$ 	&	$4(2m+1)q$ if $2\mid q$		&	$2q$ if $2\mid q$			&	$D_{2m+1}$	\\
														&						&	$2(2m+1)q$ if $2\nmid q$	&	$q$ if $2\nmid q$			&				\\
	\hline
	$(\mu_{6q} \mid \mu_{2q} ; A_{4}' \mid D_{2}')$		&	$q=4,8$							&	$24q$		&	$2q$			&	$A_{4}$		\\
	\hline
	$\mu_{2q}A_{4}'$									&	$4|q$							&	$24q$		&	$2q$			&	$A_{4}$		\\
	\hline
	$\mu_{2q}S_{4}'$									&	$q=3,8,9,16$ or $24|q$			&	$48q$		&	$2q$			&	$S_{4}$		\\
	\hline
	$(\mu_{4q} \mid \mu_{2q} ; S_{4}' \mid A_{4}')$		&	$12|q$, $24\nmid q$				&	$48q$		&	$2q$			&	$S_{4}$		\\
	\hline
	$\mu_{2q}A_{5}'$									&	$q=4,8,12,16,20,24,36,40$ or $60|q$	&	$120q$	&	$2q$			&	$A_{5}$		\\
	\hline
\end{tabular}\end{center}

\vspace{0.5cm}

\subsection*{Notation} $K$ is any algebraically closed field of characteristic $0$. The families of groups in the first column are defined in \S\ref{sec:notation} as subgroups $\GL_{2}(K)$. Every non-abelian finite group with a faithful $2$-dimensional representation is isomorphic to at least one in the table, sometimes to a few of them e.g. $\mu_{2q}A_{4}'\simeq (\mu_{6q} \mid \mu_{2q} ; A_{4}' \mid D_{2}')$ if $q$ is prime with $3$. The second column indicates which groups in each family generate a singularity of type $\rR$. To check whether a group is of type $\rR_{2}$, one has to check the possible embeddings in $\GL_{2}(K)$, e.g. $C_{16}\times A_{4}'\simeq\mu_{32}A_{4}'\simeq (\mu_{96} \mid \mu_{32} ; A_{4}' \mid D_{2}')$ has an embedding inducing an $\rR$-singularity and one inducing a $\neg \rR$-singularity.

Each family is indexed by one or two positive integers $m,q\ge 1$, we write $q'=q/\gcd(m,q)$, $m'=m/\gcd(m,q)$. A pair of positive integers $(n,d)$ is \emph{critical} if $d^{2}\cong 1\pmod{n}$, $2|n$ and $d\cong \pm1\pmod{2^{a}}$, where $2^{a}$ is the largest power of $2$ dividing $n$, or equivalently if $n$ is even and $d\cong\pm 1$ modulo every prime power dividing $n$.

\end{landscape}

\bibliographystyle{amsalpha}
\bibliography{tqs2}

\end{document}